\documentclass[oneside,leqno,10pt]{article}
\usepackage{amssymb,amsmath,latexsym,amsthm,stmaryrd}

\def\ga{\mathfrak{a}}
\def\gb{\mathfrak{b}}

\def\gg{\mathfrak{g}}
\def\gh{\mathfrak{h}}

\def\gk{\mathfrak{k}}
\def\gl{\mathfrak{l}}
\def\gm{\mathfrak{m}}
\def\gn{\mathfrak{n}}

\def\gp{\mathfrak{p}}
\def\gq{\mathfrak{q}}

\def\gs{\mathfrak{s}}
\def\gt{\mathfrak{t}}
\def\gu{\mathfrak{u}}
\def\gv{\mathfrak{v}}

\def\gz{\mathfrak{z}}

\def\gsl{\mathfrak{sl}}
\def\gso{\mathfrak{so}}

\def\gsp{\mathfrak{sp}}
\def\ggl{\mathfrak{gl}}

\def\C{\mathbb{C}}

\def\F{\mathbb{F}}

\def\H{\mathbb{H}}

\def\R{\mathbb{R}}

\def\cA{\mathcal{A}}

\def\cC{\mathcal{C}}

\def\cF{\mathcal{F}}

\def\cH{\mathcal{H}}

\def\cO{\mathcal{O}}

\def\cU{\mathcal{U}}

\def\Ad{{\rm Ad}\,}

\def\Pf{{\rm Pf}\,}

\def\Ind{{\rm Ind\,}}

\newtheorem{theorem}[equation]{Theorem}

\newtheorem{lemma}[equation]{Lemma}

\newtheorem{corollary}[equation]{Corollary}

\newtheorem{proposition}[equation]{Proposition}

\newtheorem{definition}[equation]{Definition}

\newtheorem{remark}[equation]{Remark}

\def\sideremark#1{\ifvmode\leavevmode\fi\vadjust{\vbox to0pt{\vss
 \hbox to 0pt{\hskip\hsize\hskip1em
\vbox{\hsize2cm\tiny\raggedright\pretolerance10000 
 \noindent #1\hfill}\hss}\vbox to8pt{\vfil}\vss}}} 

\title{Stepwise Square Integrable Representations \\ 
	for Locally Nilpotent Lie Groups}

\author{Joseph A. Wolf\footnote{Research partially supported by the Simons
Foundation.}}
\date{February 16, 2014}

\begin{document}

\maketitle

\abstract{In a recent paper we found conditions for a nilpotent Lie
group $N$ to have a filtration by normal subgroups whose successive quotients
have square integrable representations, and such that these square integrable
representations fit together nicely to give an explicit construction of
Plancherel almost all representations of $N$.
That resulted in explicit character formulae, Plancherel
formulae and multiplicity formulae.  We also showed that nilradicals $N$ of
minimal parabolic subgroups $P = MAN$ enjoy that ``stepwise square integrable''
property.  Here we extend those results to direct limits of
stepwise square integrable nilpotent Lie groups.  This involves some
development of the corresponding Schwartz spaces.  The main result is an
explicit Fourier inversion formula for that class of infinite dimensional
Lie groups.  One important consequence is the Fourier inversion formula for
nilradicals of classical minimal parabolic subgroups of finitary real
reductive Lie groups such as $GL(\infty;\R)$, $Sp(\infty;\C)$ and
$SO(\infty,\infty)$.}

\section{Introduction}
\label{sec1}
\setcounter{equation}{0}
A connected simply connected Lie group $N$
with center $Z$ is called {\em square integrable} if it has unitary
representations $\pi$ whose coefficients $f_{u,v}(x) = 
\langle u, \pi(x)v\rangle$ satisfy $|f_{u,v}| \in L^2(N/Z)$.  
C.C. Moore and the author worked out the structure and representation
theory of these groups \cite{MW1973}.  If $N$ has one 
such square integrable representation then there is a certain polynomial
function $\Pf(\lambda)$ on the linear dual space $\gz^*$ of the Lie algebra of
$Z$ that is key to harmonic analysis on $N$.  Here $\Pf(\lambda)$ is the
Pfaffian of the antisymmetric bilinear form on $\gn / \gz$ given by
$b_\lambda(x,y) = \lambda([x,y])$.  The square integrable
representations of $N$ are the 
$\pi_\lambda$ where $\lambda \in \gz^*$ with $\Pf(\lambda) \ne 0$,
Plancherel  almost irreducible unitary representations of $N$ are square
integrable, and up to an explicit constant 
$|\Pf(\lambda)|$ is the Plancherel density of the unitary
dual $\widehat{N}$ at $\pi_\lambda$.  
This theory has proved to have serious analytic consequences.  For example,
for most commutative nilmanifolds $G/K$, i.e. Gelfand pairs $(G,K)$ 
where a nilpotent subgroup $N$ of $G$ acts transitively on $G/K$, the
group $N$ has square integrable representations \cite{W2007}.
And it is known just which maximal parabolic subgroups of semisimple Lie groups
have square integrable nilradical \cite{W1979}.
\medskip

In \cite{W2012} and \cite{W2013} the theory of square integrable nilpotent
groups was extended to ``stepwise
square integrable'' nilpotent groups.
They are the connected simply connected nilpotent Lie groups
of (\ref{setup}) just below.  We use $L$ and $\gl$ to avoid conflict of
notation with the $M$ and $\gm$ of minimal parabolic subgroups.
\begin{equation}\label{setup}
\begin{aligned}
N = &L_1L_2\dots L_{m-1}L_m \text{ where }\\
 &\text{(a) each factor $L_r$ has unitary representations with coefficients in
$L^2(L_r/Z_r)$,} \\
 &\text{(b) each } N_r := L_1L_2\dots L_r \text{ is a normal subgroup of } N
   \text{ with } N_r = N_{r-1}\rtimes L_r \text{ semidirect,}\\
 &\text{(c) decompose }\gl_r = \gz_r + \gv_r \text{ and } \gn = \gs + \gv
        \text{ as vector direct sums where } \\
 &\phantom{XXXX}\gs = \oplus\, \gz_r \text{ and } \gv = \oplus\, \gv_r;
    \text{ then } [\gl_r,\gz_s] = 0 \text{ and } [\gl_r,\gl_s] \subset \gv
        \text{ for } r > s\,.
\end{aligned}
\end{equation}
Denote
\begin{equation}\label{c-d}
\begin{aligned}
&\text{(a) }d_r = \tfrac{1}{2}\dim(\gl_r/\gz_r) \text{ so }
        \tfrac{1}{2} \dim(\gn/\gs) = d_1 + \dots + d_m\,,
        \text{ and } c = 2^{d_1 + \dots + d_m} d_1! d_2! \dots d_m!\\
&\text{(b) }b_{\lambda_r}: (x,y) \mapsto \lambda([x,y])
        \text{ viewed as a bilinear form on } \gl_r/\gz_r \\
&\text{(c) }S = Z_1Z_2\dots Z_m = Z_1 \times \dots \times Z_m \text{ where } Z_r
        \text{ is the center of } L_r \\
&\text{(d) }\Pf: \text{ polynomial } \Pf(\lambda) = \Pf_{\gl_1}(b_{\lambda_1})
        \Pf_{\gl_2}(b_{\lambda_2})\dots \Pf_{\gl_m}(b_{\lambda_m}) \text{ on } 
	\gs^* \\
&\text{(e) }\gt^* = \{\lambda \in \gs^* \mid \Pf(\lambda) \ne 0\} \\
&\text{(f) } \pi_\lambda \in \widehat{N} \text{ where } \lambda \in \gt^*:
    \text{ irreducible unitary rep. of } N = L_1L_2\dots L_m
\end{aligned}
\end{equation}
The basic result for these groups is

\begin{theorem}\label{plancherel-general} {\rm \cite[Theorem 6.16]{W2013}}
Let $N$ be a connected simply connected nilpotent Lie group that
satisfies {\rm (\ref{setup})}.  Then Plancherel measure for $N$ is
concentrated on $\{\pi_\lambda \mid \lambda \in \gt^*\}$.
If $\lambda \in \gt^*$, and if $u$ and $v$ belong to the
representation space $\cH_{\pi_\lambda}$ of $\pi_\lambda$,  then
the coefficient $f_{u,v}(x) = \langle u, \pi_\nu(x)v\rangle$
satisfies
\begin{equation}\label{sq-orthogrel}
||f_{u,v}||^2_{L^2(N / S)} = \frac{||u||^2||v||^2}{|\Pf(\lambda)|}\,.
\end{equation}
The distribution character $\Theta_{\pi_\lambda}$ of $\pi_{\lambda}$ satisfies
\begin{equation}\label{def-dist-char}
\Theta_{\pi_\lambda}(f) = c^{-1}|\Pf(\lambda)|^{-1}\int_{\cO(\lambda)}
        \widehat{f_1}(\xi)d\nu_\lambda(\xi) \text{ for } f \in \cC(N)
\end{equation}
where $\cC(N)$ is the Schwartz space, $f_1$ is the lift
$f_1(\xi) = f(\exp(\xi))$ of $f$ from $N$ to $\gn$, 
$\widehat{f_1}$ is its classical Fourier transform,
$\cO(\lambda)$ is the coadjoint orbit $\Ad^*(N)\lambda = \gv^* + \lambda$,
$c = 2^{d_1 + \dots + d_m} d_1! d_2! \dots d_m!$ as in {\rm (\ref{c-d}a)},
and $d\nu_\lambda$ is the translate of normalized Lebesgue measure from
$\gv^*$ to $\Ad^*(N)\lambda$.  The Fourier inversion formula on $N$ is
\begin{equation}
f(x) = c\int_{\gt^*} \Theta_{\pi_\lambda}(r_xf) |\Pf(\lambda)|d\lambda
        \text{ for } f \in \cC(N).
\end{equation}
\end{theorem}
\begin{definition}\label{stepwise2}
{\rm The representations $\pi_\lambda$ of (\ref{c-d}(f)) are the
{\it stepwise square integrable} representations of $N$ relative to
the decomposition (\ref{setup}).}\hfill $\diamondsuit$
\end{definition}

One of the main results of \cite{W2012} and \cite{W2013} is that nilradicals
of minimal parabolic subgroups of finite dimensional real reductive Lie
groups are stepwise square integrable.  Even the
simplest case, the case of a minimal parabolic in $SL(n;\R)$, was a big
improvement over earlier results on the group of strictly upper triangular
real matrices.  Here we extend the construction of stepwise square integrable
representations to a class of locally nilpotent groups that are direct
limits in a manner that respects the basic setup (\ref{setup}) of the
finite dimensional case, and we show how this applies to the nilradicals
of direct limit minimal parabolic subgroups of the real and complex finitary
reductive Lie groups, including $GL(\infty;\F)$, $SL(\infty;\F)$, $U(p,q;\F)$
and $SU(p,q;\F)$ ($\F = \R, \C$ or $\H$ and $p+q = \infty$), $Sp(\infty;\F)$ 
($\F = \R$ or $\C$), and $SO^*(2\infty)$.
\medskip

In Section \ref{sec2} we examine strict direct systems $\{N_n, \varphi_{m,n}\}$
of finite dimensional connected and simply connected nilpotent Lie groups
that satisfy (\ref{setup}) in a manner that respects the maps
$\varphi_{m,n}: N_n \to N_m$\, $(m \geqq n)$.  We show how this 
leads to sequences $\{\pi_{\gamma_n}\}$ of closely related stepwise square
integrable representations of the groups $N_n$\,,  and then to their
unitary representation limits $\pi_\gamma = \varprojlim \pi_{\gamma_n}$\,.
\medskip

In Section \ref{sec3}
we prove stepwise Frobenius-Schur orthogonality relations and restriction
theorems for the coefficients of the representations $\pi_{\gamma_n}$\,.
\medskip

In Section \ref{sec4} we apply the tools of Section \ref{sec3} to obtain
inverse systems, by restriction, of the spaces
$\cA(\pi_{\gamma_n})$ of coefficients of the representations $\pi_{\gamma_n}$\,.
Then we combine density of $\cA(\pi_{\gamma_n})$ in
$\cH_{\pi_{\gamma_n}} \widehat{\otimes} \cH_{\pi_{\gamma_n}}^*$ with the
renormalization method of \cite{W2010} to construct inverse systems,
in the Hilbert space category, of the
$\cH_{\pi_{\gamma_n}} \widehat{\otimes} \cH_{\pi_{\gamma_n}}^*$\,.  These mirror the
inverse systems of the $\cA(\pi_{\gamma_n})$, resulting in an
interpretation of the function space $\cA(\pi_\gamma) = 
\varprojlim \cA(\pi_{\gamma_n})$ as
a dense subspace of the Hilbert space $\cH_{\pi_\gamma}
\widehat{\otimes} \cH_{\pi_\gamma}^* =
\varprojlim \cH_{\pi_{\gamma_n}} \widehat{\otimes} \cH_{\pi_{\gamma_n}}^*$\,.
This is somewhat analogous to the infinite dimensional
Peter---Weyl Theorem of \cite[Section 4]{W2009}.
\medskip

In Section \ref{sec5} we set up the Schwartz space machinery that will allow us 
to carry over the somewhat abstract $\cH_{\pi_\gamma}
\widehat{\otimes} \cH_{\pi_\gamma}^* =
\varprojlim \cH_{\pi_{\gamma_n}} \widehat{\otimes} \cH_{\pi_{\gamma_n}}^*$
to an explicit Fourier inversion formulae.  This, incidentally, strengthens
the stepwise $L^2$ property for coefficients involving $C^\infty$ vectors
from $L^2$ to $L^1$.
\medskip

In Section \ref{sec6} we work out that formula for the direct limit
group $N = \varinjlim N_n$\,.  See Theorem \ref{limit-inversion}.
\medskip

In Section \ref{sec7} we discuss direct systems $\{G_n , \varphi_{m,n}\}$
of finite dimensional real reductive Lie groups, and conditions on their
restricted root systems $\Delta(\gg_n,\ga_n)$, that lead to an appropriate 
limit restricted root system
$\Delta(\gg,\ga) = \varprojlim \Delta(\gg_n,\ga_n)$ of the Lie
algebra of $G = \varinjlim \{G_n , \varphi_{m,n}\}$.  That describes the
stepwise square integrable structure of the nilradicals of minimal parabolic
subgroups.  
\medskip

Finally, in Section \ref{sec8}, we arrive at the goal of this paper,
Theorem \ref{inversion-for-ss-limits},
an explicit Fourier inversion formula for the classical 
direct limit of the nilradicals of those minimal parabolics. 
\medskip

I thank Michael Christ for useful discussions of Schwartz spaces related
to the Heisenberg group.

\section{Alignment and Construction}
\label{sec2}
\setcounter{equation}{0}
For our direct limit considerations it will be necessary to adjust the
decompositions (\ref{setup}) of the connected simply connected nilpotent
Lie groups $N_n$\,.  This is so that the adjusted decompositions will
fit together as $n$ increases.  We do that by reversing the indices and
keeping the $L_r$ constant as $n$ goes to infinity.  First, we suppose that
\begin{equation}\label{nil-direct-system}
\{N_n\} \text{ is a strict direct system of connected simply connected 
	nilpotent Lie groups,}
\end{equation}
in other words  the connected simply connected nilpotent Lie groups
$N_n$ have the property that $N_n$ is a closed analytic subgroup of $N_\ell$
for all $\ell \geqq n$.  As usual, $Z_r$ denotes the center of $L_r$\,.
		For each $n$, we require that
\begin{equation}\label{newsetup}
\begin{aligned}
N_n = L_1&L_2\cdots L_{m_n} \text{ where }\\
 \text{(a) }&L_r \text{ is a closed analytic subgroup of } N_n \text{ for }
	1 \leqq r \leqq m_n\,; \\
 \text{(b) }&\text{each factor $L_r$ has unitary representations with 
	coefficients in $L^2(L_r/Z_r)$;} \\
 \text{(x) }&\text{let $L_{p,q} = L_{p+1}L_{p+2}\cdots L_q$  ($p < q$)
	and $N_{\ell,n} = L_{m_\ell +1}L_{m_\ell +2}\cdots L_{m_n}
	= L_{m_\ell,m_n}$ ($\ell < n$);}\\
 \text{(c) }&\text{each } N_{\ell ,n} 
	\text{ is a normal subgroup of } N_n 
   \text{ and } N_n = N_r \ltimes N_{r+1,n} \text{ semidirect product}; \\
 \text{(d) }&\text{decompose }\gl_r = \gz_r + \gv_r \text{ and } 
	\gn_n = \gs_n + \gu_n
        \text{ as vector space direct sums where } \\
 &\gs_n = {\bigoplus}_{r \leqq m_n}\, \gz_r \text{ and } 
	\gu_n = {\bigoplus}_{r \leqq m_n}\, \gv_r;
    \text{ then } [\gl_r,\gz_s] = 0 \text{ and } [\gl_r,\gl_s] \subset \gv
        \text{ for } r < s\,.
\end{aligned}
\end{equation}
With this setup we can follow the lines of the constructions in
\cite[Section 5]{W2013}.
\medskip

We have the Pfaffian polynomials on the $\gz_r^*$ and on $\gs_n^*$
as follows.  Given $\lambda_r \in \gz_r^*$,
extended to an element of $\gl_r^*$ by $\lambda_r(\gv_r) = 0$,
we have the antisymmetric bilinear form $b_{\lambda_r}$ on $\gl_r/\gz_r$
defined as usual by $b_{\lambda_r}(x,y) = \lambda_r([x,y])$, and
$\Pf_r(\lambda_r)$ denotes its Pfaffian.  If $\gamma_n = \lambda_1 +
\cdots + \lambda_{m_n} \in \gs_n^*$ with each $\lambda_r \in \gz_r^*$, then
we have the product
\begin{equation}
P_n(\gamma_n) = \Pf_1(\lambda_1)\Pf_2(\lambda_2) \cdots \Pf_{m_n}(\lambda_{m_n})
\end{equation}
and the nonsingular set
\begin{equation}
\gt_n^* = \{\gamma_n \in \gs_n^* \mid P_n(\gamma_n) \ne 0\}.
\end{equation}

Recall the construction (\cite{W2013}) of stepwise square integrable 
representations $\pi_{\gamma_n}$ of $N_n$\,, where $\gamma_n \in \gt_n^*$\,,
and where we adjust the indices to our situation.
If $m_n = 1$ then $\pi_{\gamma_n}$ is just the square integrable representation
$\pi_{\lambda_1}$ of $L_1$ defined by $\gamma_n = \lambda_1$\,.  Now let
$m_n > 1$ and use $N_n = (L_1L_2\dots L_{m_n-1}) \ltimes L_{m_n}
= L_{0,m_n-1} \ltimes L_{m_n}$\,.
By induction on $m_n$ we have the stepwise square integrable representation 
$\pi_{\lambda_1 + \cdots + \lambda_{m_n-1}}$ of $L_{0,m_n-1}$\,, and we view it
as a representation of $N_n$ whose kernel contains $L_{m_n}$\,.  We also
have the square integrable representation $\pi_{\lambda_{m_n}}$ of $L_{m_n}$\,.
Write $\pi'_{\lambda_{m_n}}$ for the extension of $\pi_{\lambda_{m_n}}$ to
a unitary representation of $N_n$ on the same Hilbert space 
$\cH_{\pi_{\lambda_{m_n}}}$ (the Mackey obstruction vanishes).  Now
\begin{equation}\label{def-sq-reps}
\pi_{\gamma_n} = \pi_{\lambda_1 + \cdots + \lambda_{m_n-1}} \widehat{\otimes}
        \pi'_{\lambda_{m_n}}\,.
\end{equation}
\medskip

The parameter space for our representations of the direct limit Lie group
$N = \varinjlim N_n$ will be
\begin{equation}
\gt^* = \bigcup_{n > 0} \left\{\gamma = \sum \gamma_\ell \in \gs^* \mid
	\gamma_\ell \in \gt_\ell^* \text{ for } \ell \leqq n
	\text{ and } \gamma_\ell = 0 \in \gs_\ell^*\text{ for } 
	\ell > n\right\}
\text{ where } \gs^* := \sum_{\ell > 0} \gs_\ell^* \,.
\end{equation}
The representations $\pi_\gamma$ of $N$ are defined in a manner similar 
to that of (\ref{def-sq-reps}).  Given $\gamma = \sum \gamma_\ell \in \gt^*$ 
we have the index $n = n(\gamma)$ defined by
$\gamma_\ell \in \gt_\ell^* \text{ for } \ell \leqq n(\gamma)
        \text{ and } \gamma_\ell = 0 \in \gs_\ell^*\text{ for }
        \ell > n(\gamma)$.
\begin{equation}
N = N_{n(\gamma)} \ltimes N_{n(\gamma),\infty} \text{ semidirect product, where }
	N_{n(\gamma),\infty} = {\prod}_{q > m_{n(\gamma)}} L_q\,.
\end{equation}
In particular the closed normal subgroup $N_{n(\gamma),\infty}$ satisfies
$N_{n(\gamma)} \cong N/N_{n(\gamma),\infty}$, and we denote
\begin{equation} \label{def-sq-lim-reps}
\pi_\gamma \text{ is the lift to $N$ of the stepwise square integrable
representation } \pi_{\gamma_1 + \cdots + \gamma_{m_{n(\gamma)}}} \text{ of } 
N_{n(\gamma)}\,.
\end{equation}
The representation space of $\pi_\gamma$ is the projective (jointly continuous)
tensor product
\begin{equation}\label{gamma-decomp}
\cH_{\pi_\gamma} = \cH_{\pi_{\gamma_1}} \widehat{\otimes} 
	\cH_{\pi_{\gamma_2}}
	\widehat{\otimes} \cdots \widehat{\otimes} 
	\cH_{\pi_{\gamma_{n(\gamma)}}}
\end{equation}

These representations $\pi_\gamma$ are the {\sl limit stepwise square 
integrable} representations of $N$.  We go on the see the extent to which
their coefficients and characters imitate the properties of Theorem
\ref{plancherel-general}.

\section{Coefficient Functions}
\label{sec3}
\setcounter{equation}{0}
Let $\cH_{\pi_\gamma}$ denote the representation space of $\pi_\gamma$ and
$\langle\,\,\cdot\, , \,\cdot\,\,\rangle_{\pi_\gamma}$ the hermitian inner 
product on $\cH_{\pi_\gamma}$\,.  Given $u, v \in \cH_{\pi_\gamma}$ we have the
coefficient function on $N$ given by
\begin{equation}
f_{\pi_\gamma,u,v}(g) = \langle u, \pi_\gamma(g)v\rangle_{\pi_\gamma}\,.
\end{equation}
We use the standard $(r(x)f)(g) = f(gx)$ and $(\ell(y)f)(g) = f(y^{-1}g)$.  
These right and left translations commute with each other.  They are 
well defined on the $f_{\pi_\gamma,u,v}$ and satisfy
\begin{equation}
\ell(x) r(y): f_{\pi_\gamma,u,v} \mapsto f_{\pi_\gamma, \pi_\gamma(x)u,
	\pi_\gamma(y)v}\,.
\end{equation}
By our construction (\ref{def-sq-lim-reps}), the value 
$f_{\pi_\gamma,u,v}(g)$ depends only on the coset $gN''_{n(\gamma)}$.  In
other words it really is a function on $N_{n(\gamma)} \cong N/N''_{n(\gamma)}$. 
Further, $|f_{\pi_\gamma,u,v}(g)|$ depends only on the coset 
$gS_{n(\gamma)}N''_{n(\gamma)}$ where $S_{n(\gamma)}$ is the quasicenter
$Z_1Z_2\cdots Z_{m_{n(\gamma)}}$
of $N_{n(\gamma)} = L_1L_2\cdots L_{m_{n(\gamma)}}$\,.
Building on (\ref{sq-orthogrel}), we have the following variation on the 
Frobenius-Schur orthogonality relations for finite groups:
\begin{proposition}\label{sq-orthogrel-quo}
Let $\gamma \in \gt^*$ and $n = n(\gamma)$.  Then
$||f_{\pi_\gamma,u,v}||^2_{L^2(N / S_n N''_n)} 
	= \frac{||u||_{\pi_\gamma}^2||v||_{\pi_\gamma}^2}{|P_n(\gamma)|}$\,\,.
\end{proposition}

\begin{proof}
This is an induction on $n$.  The case $n = 1$ is (\ref{sq-orthogrel}).
Now go from $n$ to $n+1$.  Express $N_{n+1} = N_n \ltimes N_{n,n+1}$ where
$$
N_n = L_1L_2\cdots L_{m_n} \text{ and }
	N_{n,q} = L_{m_n+1}L_{m_n+2}\cdots L_{m_q}
	\text{ for } q > n.
$$
Then $S_{n+1} = S_n \times S_{n,n+1}$ where the quasi-centers
$$
S_n = Z_1Z_2\cdots Z_{m_n} \text{ and }
	S_{n,q} = Z_{m_n+1}Z_{m_n+2}\cdots Z_{m_q}
	\text{ for } q > n.
$$
Now let $\gamma_n \in \gt_n^*$ and $\gamma_{n,n+1} \in \gt_{n,n+1}^*$
where, as before, $\gt^*$ is the nonzero set of the Pfaffian in $\gs^*$.
Note that $\pi_{\gamma_n} \in \widehat{N_n}$ and $\pi_{\gamma_{n,n+1}}
\in \widehat{N_{n,n+1}}$ are stepwise square integrable.  Write
$\pi'_{\gamma_{n,n+1}}$ for the extension of $\pi_{\gamma_{n,n+1}}$ from
$N_{n,n+1}$ to $N_{n+1}$\,.  Let
$u,v \in \cH_{\pi_{\gamma_n}}$ and $x,y \in \cH_{\pi_{\gamma_{n,n+1}}}$
so $u\otimes x, v \otimes y \in \cH_{\pi_{\gamma_{n+1}}}$\,.  Let
$a$ run over $N_n$ and let $b$ run over $N_{n,n+1}$\,.  Compute
$$
\begin{aligned}
&||f_{\pi_{\gamma_{n+1}},u\otimes x,v\otimes y}||^2_{L^2(N_{n+1}/S_{n+1})} \\
&\phantom{XX}= \int_{N_{n+1}/S_{n+1}} |\langle u\otimes x, (\pi_{\gamma_n}
	\widehat{\otimes} \pi'_{\gamma_{n,n+1}})(ab)(v\otimes y)\rangle|^2
	da\,db\\
&\phantom{XX}= \int_{N_{n+1}/S_{n+1}} |\langle u\otimes x,
	\pi_{\gamma_n}(a)\pi'_{\gamma_{n,n+1}}(b)(v) \otimes
	\pi_{\gamma_{n,n+1}}(b)(y)\rangle|^2 da\,db \\
&\phantom{XX}= \int_{N_{n,n+1}/S_{n,n+1}} \left ( \int_{N_n/S_n}
	|\langle u, \pi_{\gamma_n}(a)\pi'_{\gamma_{n,n+1}}(b)(v)\rangle |^2
	da \right )
	\cdot |\langle x, \pi_{\gamma_{n,n+1}}(b)(y)\rangle|^2 db \\
&\phantom{XX}= \int_{N_{n,n+1}/S_{n,n+1}} 
	\frac{||u||^2 ||\pi'_{\gamma_{n,n+1}}(b)(v)||^2}{|\Pf_n(\gamma_n)|}
	|\langle x, \pi_{\gamma_{n,n+1}}(b)(y)\rangle|^2 db \\
&\phantom{XX}= \int_{N_{n,n+1}/S_{n,n+1}} 
        \frac{||u||^2 ||v||^2}{|\Pf_n(\gamma_n)|}
        |\langle x, \pi_{\gamma_{n,n+1}}(b)(y)\rangle|^2 db \\
&\phantom{XX}= \frac{||u||^2 ||v||^2}{|\Pf_n(\gamma_n)|}\cdot
	\frac{||x||^2 ||y||^2}{|\Pf_{n,n+1}(\gamma_{n,n+1})|} = 
	\frac{||u\otimes x||^2 ||v \otimes y||^2}{|\Pf_{n+1}(\gamma_{n+1})|}.
\end{aligned}
$$
The proposition follows.
\end{proof}

In the notation of the proof of Proposition \ref{sq-orthogrel-quo},
\begin{equation}\label{restr}
f_{\pi_{\gamma_{n+1}},u\otimes x,v\otimes y}(a) 
        = \langle u, \pi_{\gamma_n}(a)v\rangle\cdot \langle x,y\rangle
        = \langle x,y\rangle f_{\pi_{\gamma_n},u,v}(a)
        \text{ for } a \in N_n\,.
\end{equation}
In other words, $f_{\pi_{\gamma_{n+1}},u\otimes x,v\otimes y}|_{N_n}
= \langle x,y\rangle f_{\pi_{\gamma_n},u,v}$\,.  In particular the case
where $x = e = y$, where $e$ is a unit vector, is
\begin{equation}\label{restr1}
f_{\pi_{\gamma_{n+1}},u\otimes e,v\otimes e}|_{N_n} = f_{\pi_{\gamma_n},u,v}
\end{equation}
\medskip

Iterating this and combining it with Proposition \ref{sq-orthogrel-quo}
we arrive at

\begin{proposition}\label{rescale}
Let $\gamma \in \gt^*$ and $n = n(\gamma)$.  Let $\gamma' \in \gt^*$
and $n' = n(\gamma')$ with $n' > n$ and $\gamma'|_{\gs_n} = \gamma$.  Then 
$\pi_{\gamma'}|_{N_n}$ is an infinite multiple of $\pi_\gamma$.
Split $\cH_{\pi_{\gamma'}} = \cH_{\pi_\gamma} \widehat{\otimes}\cH''$ where
$\cH'' = \cH_{\pi_{\gamma'_{n+1}}} \widehat{\otimes} \cdots \widehat{\otimes}
\cH_{\pi_{\gamma'_{n'}}}$ in the notation of {\rm (\ref{gamma-decomp})}.  
Choose a unit vector $e \in \cH''$.  Then
\begin{equation}\label{unitfactor}
\cH_{\pi_{\gamma}} \hookrightarrow \cH_{\pi_{\gamma'}} \text{ by }
v \mapsto v\otimes e
\end{equation}
is a well defined $N_n$--equivariant isometric injection.  If
$u, v \in \cH_{\pi_\gamma}$ then
\begin{equation}\label{rescale1}
||f_{\pi_{\gamma'},u\otimes e,v\otimes e}||^2_{L^2(
	N / S_{n'} N''_{n'})} 
	= \frac{|P_n(\gamma)|}{|P_{n'}(\gamma')|}
	||f_{\pi_\gamma,u,v}||^2_{L^2(N / S_n N''_n)}\,.
\end{equation}
\end{proposition}

\section{Hilbert Space Limits}
\label{sec4}
\setcounter{equation}{0}
Now we combine the restriction maps of Section \ref{sec3}.
Let $\gamma \in \gt^*$ and $n = n(\gamma)$.  Then $\gamma$ defines a
unitary character
\begin{equation}\label{defrelchar}
\zeta_\gamma = \exp(2\pi i \gamma) \text{ by } \zeta_\gamma(\exp(\xi)y) = 
	e^{2\pi i \gamma(\xi)} \text{ where } \xi \in \gs_n \text{ and }
	y \in N''_n\,.
\end{equation}
That defines the Hilbert space
\begin{equation}\label{defrelHilbert}
\begin{aligned}
L^2&(N/S_nN''_n,\zeta_\gamma) = \\
&\{f:N \to \C \mid f(gx) = \zeta_\gamma(x)^{-1}
	f(g) \text{ and } |f| \in L^2(N / S_n N''_n) \text{ for all }
	g \in N \text{ and } x \in N''_n\}.
\end{aligned}
\end{equation}
The finite linear combinations of the coefficients 
$f_{\pi_\gamma, u, v}$ (where $u, v \in \cH_{\pi_\gamma}$) form a dense
subspace of $L^2(N/S_nN''_n,\zeta_n$), and that gives an
$N \times N$ equivariant Hilbert space isomorphism
\begin{equation}
L^2(N/S_nN''_n,\zeta_\gamma) \cong 
	\cH_{\pi_\gamma} \widehat{\otimes} \cH_{\pi_\gamma}^*.
\end{equation}
\medskip

We know that the stepwise square integrable group $N_n = N/N_n''$ satisfies
\begin{equation}
L^2(N_n) = L^2(N/N_n'') 
= \int_{\gamma \in \gt^*\,\,  and\,\, n(\gamma) = n}
   (\cH_{\pi_\gamma} \widehat{\otimes} \cH_{\pi_\gamma}^*) |P_n(\gamma)|d\gamma.
\end{equation}
In brief, that expands the functions on $N$ that depend only on the first
$m(n)$ factors in $N = N_1N_2N_3\cdots$\,.  To expand the functions that
depend on more factors, say the first $m(n')$ factors in the notation of
Proposition \ref{rescale}, we would like to inject
$$L^2(N/N_n'') = 
  \int_{\gamma \in \gt_n^*} L^2(N/S_nN''_n,\zeta_\gamma)|P_n(\gamma)|d\gamma
$$
into
$$ 
L^2(N/N_{n'}'') = 
  \int_{\gamma' \in \gt_{n'}^*}L^2(N/S_{n'}N''_{n'},\zeta_{\gamma'})
	|P_{n'}(\gamma')|d\gamma'
$$
using the renormalizations of (\ref{rescale1}).  However, $\gamma$ has
many extensions $\gamma'$ with the given $n(\gamma') = n'$, so this will
not work directly.  But we can take the orthogonal projections dual to
the injections of (\ref{rescale1}) and form an inverse system of Hilbert
spaces.  
\medskip

To start, if $u, v \in \cH_{\pi_\gamma}$ and $x, y \in \cH''$, using
(\ref{restr}) and Proposition \ref{rescale}, 
\begin{equation}\label{rescale2}
p_{\gamma',\gamma}: f_{\pi_{\gamma'},u\otimes x,v\otimes y} \mapsto
\langle x, y \rangle \left |\frac{P_n(\gamma)}{P_{n'}(\gamma')}\right |^{1/2} 
f_{\pi_\gamma,u,v}
\end{equation}
is the orthogonal projection dual to the isometric inclusion (\ref{unitfactor}).
Since $\gamma$ is the restriction of $\gamma'$ from $\gs_{n(\gamma')}$ to
$\gs_{n(\gamma)}$ we can reformulate (\ref{rescale2}) as
\begin{equation}\label{rescale3}
p_{\gamma',n} : f_{\pi_{\gamma'},u\otimes x,v\otimes y} \mapsto
	\langle x, y \rangle 
	\left | \frac{P_n(\gamma'|_{\gs_n})}{P_{n'}(\gamma')} \right |^{1/2}
	f_{\pi_{\gamma'|_{\gs_n},u,v}} \text{ where } n = n(\gamma).
\end{equation}
The maps $p_{\gamma,n}$ of (\ref{rescale3}) sum to a Hilbert space projection,
essentially restriction of coefficients,
\begin{equation}\label{rescale4}
p_{n',n} = \left ( \int_{\gamma'\in\gs^*_{n'}}p_{\gamma',n}\,d\gamma'\right ) : 
	L^2(N_{n'}) \to
	L^2(N_n) \text{ where } n = n(\gamma'|_{\gs_n})
	\text{ and } n' = n(\gamma') \geqq n.
\end{equation}
The maps $p_{n',n}$ of (\ref{rescale4}) define an inverse system in the
category of Hilbert spaces and partial isometries:
\begin{equation}\label{rescale5}
L^2(N_1) \overset{p_{2,1}}{\longleftarrow} L^2(N_2) 
	\overset{p_{3,2}}{\longleftarrow}  L^2(N_3) 
	\overset{p_{4,3}}{\longleftarrow} \,\, ... \,\,\,\longleftarrow\,
	L^2(N)
\end{equation}
where the projective limit $L^2(N) := \varprojlim \{L^2(N_n),p_{n',n}\}$ 
is taken in the category
of Hilbert spaces and partial isometries.  We now have the Hilbert space
\begin{equation}\label{rescale6}
L^2(N) := \varprojlim \{L^2(N_n), p_{n',n}\}.
\end{equation}

\section{The Schwartz Spaces}
\label{sec5}
\setcounter{equation}{0}
In order to refine (\ref{rescale6}) to a Fourier
inversion formula we must first
make it more explicit.  The span $\cA(\pi_{\gamma_n})$ of the 
coefficients of the representation $\pi_{\gamma_n}$ is dense in
the space of functions on $N_n$ given by
$\cH_{\pi_{\gamma_n}} \widehat{\otimes} \cH_{\pi_{\gamma_n}}$\,. 
The idea in the background here is to realize Schwartz class functions
as wave packets
$
f(a) = \int_{\gs_n^*} \varphi(\gamma_n) f_{\pi_{\gamma_n},u(\gamma_n),
v(\gamma_n)}(a)d\gamma_n
$
where $\varphi$ is a Schwartz class function on $\gs_n$ and where 
$u(\gamma_n)$ and $v(\gamma_n)$ are fields of 
$C^\infty$ unit vectors in the $\cH_{\pi_{\gamma_n}}$.
More concretely we show that the coefficient
$f_{\pi_{\gamma_n,u,v}}$ belongs to an appropriate Schwartz space (and thus
an appropriate $L^1$ space) when $u$ and $v$ are $C^\infty$ vectors for
$\pi_{\gamma_n}$\,.
\medskip

We first collect some standard facts from Kirillov theory concerning the 
analog of the Schr\" odinger representation of the Heisenberg group.
Let $L$ be a connected simply connected nilpotent Lie group that has
square integrable representations.  $Z$ is the center of $L$, and
$\lambda \in \gz^*$ with $\Pf_\gl(\lambda) \ne 0$.  Let $\gp$ and $\gq$ 
be totally real polarizations for $\lambda$, $\gp = \gz + \ga$ and
$\gq = \gz + \gb$, and suppose that we chose them so that 
$b_\lambda(x,y) = \lambda([x,y])$ gives  a nondegenerate pairing of
$\ga$ with $\gb$.  In this setting, the square integrable representation
$\pi_\lambda$ of $L$ is $\Ind_P^N(\exp(2\pi i \lambda)$, and it
represents $L$ on $L^2(N/P) = L^2(B)$.  Further, here $\pi_\lambda$
maps the universal enveloping algebra $\cU(\gl)$ onto the set of all
polynomial (in linear coordinates from $\exp: \gb \to B$) differential operators
on $B$.  In particular,
\begin{lemma}\label{sch-vectors-sqint}
The $C^\infty$ vectors for the representation $\pi_\lambda$ are the
Schwartz class functions on $B$.  In other words if $p$ and $q$
are polynomials on $B$\,, if $D$ is a constant coefficient differential
operator on $B$\,, and if $u: B \to \C$ is a $C^\infty$ vector
for $\pi_\lambda$\,, then $|q(x)p(D)u|$ is bounded.
\end{lemma}

In order to extend this to stepwise square integrable representations we
must take into account the problem that $S_n$ need not be central in $N_n$\,.
We do this by decomposing 
\begin{equation}\label{decomp_n_n}
N_n \simeq L_1 \times \cdots \times L_{m(n)}
\end{equation}
where $\simeq$ is the measure preserving real analytic diffeomorphism given
by the polynomial map
\begin{equation}\label{decomp_exp_n}
\exp':\gn_n \to N_n \text{ by } \exp'(\xi_1 + \dots + \xi_{m(n)})
        = \exp(\xi_1)\exp(\xi_2)\cdots\exp(\xi_{m(n)})
\text{ where each } \xi_r \in \gl_r\,.
\end{equation}
Using the part of (\ref{newsetup}d) that says $[\gl_r,\gz_s] = 0$ for $r < s$
the decomposition (\ref{decomp_n_n}) gives us
\begin{equation}\label{decomp_n_n_mod_s}
\begin{aligned}
N_n/S_n &= \{x_{m(n)}\cdots x_2x_1Z_{m(n)}\cdots Z_2Z_1\mid x_r \in L_r\}\\
&= \{x_{m(n)}Z_{m(n)} \cdots x_2Z_2x_1Z_1\mid x_r \in L_r\}
= (L_{m(n)}/Z_{m(n)}) \times \cdots \times (L_1/S_1) \\
&\simeq (L_1/S_1) \times \cdots \times (L_{m(n)}/Z_{m(n)}).
\end{aligned}
\end{equation}
Now let $\gamma_n = \lambda_1 + \cdots + \lambda_{m(n)} \in \gt^*_n$\,.
Let $\gp_r$ and $\gq_r$ be totally real polarizations on $\gl_r$ for
$\lambda_r$\,, paired as above by $b_{\lambda_r}$\,.  We do not claim
that $\gp = \sum \gp_r$ and $\gq = \sum \gq_r$ are polarizations on $\gn_n$
for $\gamma_n$\, (we don't know that they are algebras), but still
$\gp_r = \gz_r + \ga_r$ and $\gq_r = \gz_r + \gb_r$ where $b_{\lambda_r}$ 
pairs $\ga_r$ with $\gb_r$\,, so $b_{\gamma_n}$ is
a nondegenerate pairing of $\ga = \sum \ga_r$ with $\gb = \sum \gb_r$\,.
Now the stepwise square integrable representation $\pi_{\gamma_n}$ of $N_n$
is realized on $L^2(B)$ where $B = \exp'(\gb)$ in the notation of
(\ref{decomp_exp_n}).  Again, in this setting, $\pi_{\gamma_n}$
maps the universal enveloping algebra of $\gn_n$ onto the set of all
polynomial (in linear coordinates from $\exp': \gb \to B$) 
differential operators on $B$.  This extends Lemma \ref{sch-vectors-sqint}
to
\begin{lemma}\label{sch-vectors}
Identify $B = \exp'(\gb)$ with the real vector space $\gb$.
The $C^\infty$ vectors for the representation $\pi_{\gamma_n}$ are the
Schwartz class functions on $B$.  In other words if $p$ and $q$
are polynomials on $B$\,, if $D$ is a constant coefficient differential
operator on $B$\,, and if $u: B \to \C$ is a $C^\infty$ vector
for $\pi_{\gamma_n}$\,, then $|q(x)p(D)u|$ is bounded.
\end{lemma}

Now consider the Schwartz space analog of the definition (\ref{defrelHilbert}).
We define the {\em relative Schwartz space} $\cC(N/S_nN''_n,\zeta_\gamma)
= \cC(N_n/S_n,\zeta_{\gamma_n})$ to be
\begin{equation}\label{defrelSchwartz}
\begin{aligned}
\text{ all } f\in &C^\infty(N) \text{ such that}\\
	& f(xs) = \zeta_\gamma(s)^{-1} f(x) \text{ for all }
		x \in N_n \text{ and } s \in S_n\,,\text{ and } \\
	& |q(x)p(D)f| \text{ is bounded for all polynomials } 
        p,q \text{ on } N_n/S_n \text{ and all } D \in \cU(\gn_n).
\end{aligned}
\end{equation}
It is a nuclear Fr\' echet space and is dense in $L^2(N/S_nN''_n,\zeta_\gamma)
= L^2(N_n/S_n,\zeta_{\gamma_n})$.  
\medskip

We define 
$C_c^\infty(N/S_nN''_n,\zeta_\gamma)
= C_c^\infty(N_n/S_n,\zeta_{\gamma_n})$ as the space of functions
$f \in C^\infty(N)$ such that $f(xs) = \zeta_\gamma(s)^{-1} f(x)$ for all
$x \in N_n \text{ and } s \in S_n$\,, whose absolute values are compactly
supported modulo $S_nN''_n$ for $C_c^\infty(N/S_nN''_n,\zeta_\gamma)$,
modulo $S_n$ for $C_c^\infty(N_n/S_n,\zeta_{\gamma_n})$. It is
dense in the corresponding Schwartz space.  Thus we have the expected
continuous inclusions $C_c^\infty \hookrightarrow \cC \hookrightarrow L^2$
with dense images.

\begin{theorem}\label{coef-sch}
Let $u$ and $v$ be $C^\infty$ vectors for the stepwise square integrable
representation $\pi_{\gamma_n}$ of $N_n$\,.  Define $\zeta_\gamma$ and 
$\zeta_{\gamma_n}$
as in {\rm (\ref{defrelchar})}, and $A = \exp'(\ga)$ and $B = \exp'(\gb)$
as in the discussion following {\rm(\ref{decomp_n_n_mod_s})}.  
Then the coefficient function
$f_{\pi_{\gamma_n},u,v}$ belongs to the relative Schwartz space
$\cC(N/S_nN''_n,\zeta_\gamma) = \cC(N_n/S_n,\zeta_{\gamma_n})$\,.
\end{theorem}

\begin{proof}
Write $f_{u,v}$ for $f_{\pi_{\gamma_n},u,v}$ and $\pi$ for 
$\pi_{\gamma_n}$\,.  So $f_{u,v}(x) = \langle u, \pi(x)v\rangle$.
The left/right action of the enveloping algebra is $Df_{u,v}E = 
f_{\pi(D)u,\pi(E)v}$.  View $u \in \cC(A)$ and $v \in \cC(B)$.
Here $\pi(D)u$ is the image of $u$ under the (arbitrary) polynomial
differential operator $\pi(D)$ on $A$ and $\pi(E)v$ is the
image of $v$ under the (arbitrary) polynomial differential operator
$\pi(D)$ on $B$.  Together they give the image of $f_{u,v}$ under
the polynomial differential operator $\pi(D) \otimes \pi(E)$
on $A \times B = N_n/S_n$\,.  Every polynomial differential operator
on $A \times B$ is a finite sum of such operators $\pi(D) \otimes \pi(E)$.
Since coefficients are bounded, here
$|f_{\pi(D)u,\pi(E)v}(x)| \leqq ||\pi(D)u||\cdot ||\pi(E)v||$, and
since $f_{\pi(D)u,\pi(E)v}(xs) = \zeta(s)^{-1}f_{\pi(D)u,\pi(E)v}(x)$,
the coefficient $f_{u,v} \in \cC(N_n/S_n,\zeta_{\gamma_n})$.
\end{proof}

\begin{corollary}\label{l1-coef1}
Let $u$ and $v$ be $C^\infty$ vectors for the stepwise square integrable
representation $\pi_{\gamma_n}$ of $N_n$\,.  Then the coefficient function
$f_{\pi_{\gamma_n},u,v} \in L^1(N_n/S_n,\zeta_{\gamma_n})$\,.
\end{corollary}

\begin{corollary}\label{l1-coef2}
Let $L$ be a connected simply connected nilpotent Lie group, $Z$ its center, 
and $\pi$
a square integrable representation of $L$.  Let $\zeta \in \widehat{Z}$
such that $\pi|_Z$ is a multiple of $\zeta$.  Let $u$ and $v$ be $C^\infty$
vectors for $\pi$.  Then the coefficient $f_{\pi,u,v} \in
L^1(L/Z,\zeta)$.
\end{corollary}

Any norm $|\xi|$ on $\gn_n$ carries over to a norm $|\exp(\xi)| := |\xi|$
on $N_n$.  We have the standard Schwartz space $\cC(N_n)$, given by the 
seminorms $\nu_{k,D,E}(f) = \sup_{x \in N_n} |(1 + |x|^2)^k (DfE)(x)|$
where $k$ is a positive integer and $D, E \in \cU(\gn_n)$ acting on the
left and right.  Since $\exp: \gn_n \to N_n$ is a polynomial diffeomorphism
it gives a topological isomorphism of $\cC(N_n)$ onto the classical
Schwartz space $\cC(\gn_n)$.  Fourier transform and
inverse Fourier transform of Schwartz class functions on $S_n$.

\begin{remark}\label{make-relative}
{\rm If $\gamma_n \in \gs_n^*$ and $f \in \cC(N_n)$ define
$f_{\gamma_n}(x) = \int_{S_n} f(xs)\zeta_{\gamma_n}(s)ds$.  Then
$f_{\gamma_n} \in \cC(N_n/S_n,\zeta_{\gamma_n})$.
Let $z \in S_n$\,.  Since $S_n$ is commutative, 
$$f_{\gamma_n}(xz) = \int_{S_n} f(xzs)\zeta_{\gamma_n}(s)ds
= \int_{S_n} f(xsz)\zeta_{\gamma_n}(s)ds
= \int_{S_n} f(xs)\zeta_{\gamma_n}(z^{-1}s)ds
= \zeta_{\gamma_n}(z)^{-1}f_{\gamma_n}(x).
$$
Given $x \in N_n$ we view $f_{\gamma_n}(x)$ as a function on $\gs_n^*$
by $\varphi_x(\gamma_n) := f_{\gamma_n}(x)$.  Note that $\varphi_x$ is
(a multiple of) the Fourier transform of the left translate
$(\ell(x^{-1})f)|_{S_n}$\,, say $\cF_{S_n}(\ell(x^{-1})f)|_{S_n}$.  
The inverse Fourier $\cF_{S_n}^{-1}(\varphi_x)$ transform reconstructs
$f$ from the $f_{\gamma_n}$\,.   Each of the $f_{\gamma_n}$ is a limit
(in $\cC(N_n/S_n,\zeta_{\gamma_n})$) of finite linear combinations of
coefficient functions $f_{\pi_{\gamma_n}, u,v}$\,.  Thus every
$f \in \cC(N_n)$ is approximated (in $\cC(N_n)$) by wave packets of
coefficient functions of stepwise square integrable representations. }
\hfill $\diamondsuit$
\end{remark}

Proceeding as in Section \ref{sec4} let $n' \geqq n$ and consider
$\gamma_{n'} \in \gt^*_{n'}$ with $\gamma_{n'}|_{\gs_n} = \gamma_n$.  
For brevity write $\gamma = \gamma_n$ and $\gamma' = \gamma_{n'}$\,.
We reformulate
(\ref{rescale4}) through (\ref{rescale6}) for the Schwartz spaces.
\begin{equation}\label{rescale4s}
q_{n',n}: \cC(N_{n'}) \to \cC(N_n) \text{ by } f \mapsto f|_{N_n}\,. 
\end{equation}
The maps $q_{n',n}$ of (\ref{rescale4s}) define an inverse system in the
category of complete locally convex topological vector spaces
\begin{equation}\label{rescale5s}
\cC(N_1) \overset{q_{2,1}}{\longleftarrow} \cC(N_2) 
        \overset{q_{3,2}}{\longleftarrow}  \cC(N_3) 
        \overset{q_{4,3}}{\longleftarrow}  ... 
\end{equation}
We define the projective limit
\begin{equation}\label{rescale6s}
\cC(N) := \varprojlim \{\cC(N_n),q_{n',n}\}
\end{equation}
to be the Schwartz space of $N = \varinjlim N_n$\,.  This is dual to the 
construction of \cite[(2.20)]{W2010}.  Now we relate it to (\ref{rescale6}).
We scale the natural injections to maps
\begin{equation}\label{rescalecl1}
r_{n,\gamma}: \cC(N_n/S_n,\zeta_\gamma) \to L^2(N_n/S_n,\zeta_\gamma)
\text{ by } f \mapsto |\Pf_{\gn_n}(\gamma)|^{1/2}f\,.
\end{equation}
They sum to maps
\begin{equation}\label{rescalecl2}
r_n = \left ( \int_{\gs_n^*} r_{n,\gamma}\, d\gamma\right ) : 
	\cC(N_n) \to L^2(N_n)
\end{equation}
that are equivariant for the maps $p_{n',n}$ and $q_{n',n}$\,.  The 
arguments leading to \cite[Proposition 2.22]{W2010} can be dualized
from direct limits to projective limits.  Thus, dual to
\cite[Proposition 2.22]{W2010},
\begin{proposition}\label{compare}
The maps $r_n$ of {\rm (\ref{rescalecl2})} satisfy
$p_{n',n}\cdot r_{n'} = r_n\cdot q_{n',n}$ for $n' \geqq n$ and send the
inverse system $\{\cC(N_n),q_{n',n}\}$ into the inverse system
$\{L^2(N_n),p_{n',n}\}$.  That defines a continuous $N$--equivariant injection
$$
r: \cC(N) \to L^2(N)
$$
with dense image.  In particular $r$ defines a pre Hilbert space structure on
$\cC(N)$ with completion isometric to $L^2(N)$.
\end{proposition}

\section{Fourier Inversion for the Limit Group}
\label{sec6}
\setcounter{equation}{0}
In this section we apply the material of Section \ref{sec5} to extend the
Fourier inversion portion of Theorem \ref{plancherel-general}
from the $N_n$ to the limit group $N = \varinjlim N_n$\,.
To set this up recall that
\begin{itemize}
\item $\gt^* = \varprojlim \gt^*_n$ consists of all
collections $\gamma = (\gamma_n)$ where each $\gamma_n \in \gt_n^*$ and
if $n' \geqq n$ then $\gamma_{n'}|_{\gs_n} = \gamma_n$\,.  
\item given $\gamma = (\gamma_n) \in \gt^*$ the limit representation
$\pi_\gamma = \varprojlim \pi_{\gamma_n}$ is constructed as in 
Section \ref{sec2},
\item The distribution character $\Theta_{\pi_{\gamma_n}}$ are given by
(\ref{def-dist-char}), and
\item $\cC(N) = \varprojlim \cC(N_n)$ consists of all collections
$f = (f_n)$ where each $f_n \in \cC(N_n)$ and if $n' \geqq n$ then
$f_{n'}|_{N_n} = f_n$\,.
\end{itemize}
Then the limit Fourier inversion formula is
\begin{theorem}\label{limit-inversion}
Suppose that $N = \varinjlim N_n$ where $\{N_n\}$ satisfies 
{\rm (\ref{newsetup})}.  Then the Plancherel measure for $N$ is
concentrated on $\gt^*$.   Let $f = (f_n) \in \cC(N)$ and $x \in N$.
Then $x \in N_n$ for some $n$ and
\begin{equation}\label{lim-inv-formula}
f(x) = c_n\int_{\gt_n^*} \Theta_{\pi_{\gamma_n}}(r_xf) 
	|\Pf_{\gn_n}({\gamma_n})|d\gamma_n
\end{equation}
where $c_n = 2^{d_1 + \dots + d_m} d_1! d_2! \dots d_m!$ as in {\rm (\ref{c-d}a)}
and $m$ is the number of factors $L_r$ in $N_n$. 
\end{theorem}

\begin{proof}  Apply Theorem \ref{plancherel-general} to $N_n$.  That gives
$f(x) = f_n(x) =
c_n\int_{\gt_n^*} \Theta_{\pi_{\gamma_n}}(r_xf) |\Pf_{\gn_n}({\gamma_n})|\,
d\gamma_n$\,.
\end{proof}

\begin{remark}{\rm
A Plancherel Formula of the sort $||f||_{L^2}^2 = \int ||\pi(f)||^2_{HS}\, d\pi$
usually is somewhat easier than a Fourier inversion formula.  This 
in part is 
because it usually is easier to prove that operators $\pi(f)$ are
Hilbert-Schmidt than to prove that (for appropriate functions $f$)
they are of trace class.  Thus one might expect that a formula 
$||f||^2_{L^2(N)} = \lim c_n'\int_{\gt_n^*}
||\pi_{\gamma_n}(f|_{N_n})||^2_{HS} |\Pf_{\gn_n}({\gamma_n})|\,d\gamma_n$
would be easier to prove than (\ref{lim-inv-formula}).  But it is not
clear how to relate the Hilbert--Schmidt norms to the limit process, because
we have have not yet found an appropriate form of the Frobenius--Schur 
orthogonality relations.  Thus the ``less delicate'' Plancherel Formula remains
problematical.  } \hfill $\diamondsuit$
\end{remark}

\section{Nilradicals of Parabolics in Finite Dimensional Groups}
\label{sec7}
\setcounter{equation}{0}
In Section \ref{sec8} we will specialize our results to nilradicals 
of minimal parabolic subgroups
of finitary real reductive Lie groups such as the infinite special and general
linear groups and the infinite real, complex and quaternionic unitary groups.
In order to do that, in this section we review the relevant restricted
root structure that gives the finite dimensional case, reversing some
of the enumerations used in \cite{W2013} to be appropriate for our direct limit
systems.
\medskip

Let $G$ be a finite dimensional connected 
real reductive Lie group.  We recall some structural results
on its minimal parabolic subgroups, some standard and some from \cite{W2013}.
\medskip

Fix an Iwasawa decomposition $G = KAN$.  Write $\gk$ for the Lie 
algebra of $K$, $\ga$ for the Lie algebra of $A$, and $\gn$ for the
Lie algebra of $N$.  Complete $\ga$ to a Cartan subalgebra $\gh$ of $\gg$.
Then $\gh = \gt + \ga$ with $\gt = \gh \cap \gk$.  Now we have root systems
\begin{itemize}
\item $\Delta(\gg_\C,\gh_\C)$: roots of $\gg_\C$ relative to $\gh_\C$ 
(ordinary roots), and

\item $\Delta(\gg,\ga)$: roots of $\gg$ relative to $\ga$ (restricted roots). 

\item $\Delta_0(\gg,\ga) = \{\gamma \in \Delta(\gg,\ga) \mid 
	2\gamma \notin \Delta(\gg,\ga)\}$ (nonmultipliable restricted roots).
\end{itemize}
Sometimes we will identify a restricted root
$\gamma = \alpha|_\ga$, $\alpha \in \Delta(\gg_\C,\gh_\C)$ and 
$\alpha |_\ga \ne 0$, with the set 
\begin{equation}\label{resrootset}
[\gamma] := 
\{\alpha' \in \Delta(\gg_\C,\gh_\C) \mid \alpha'|_\ga = \alpha|_\ga\}
\end{equation}
of all roots that restrict to it.  Further, 
$\Delta(\gg,\ga)$ and $\Delta_0(\gg,\ga)$ are root 
systems in the usual sense.  Any positive system 
$\Delta^+(\gg_\C,\gh_\C) \subset \Delta(\gg_\C,\gh_\C)$ defines positive 
systems
\begin{itemize}
\item $\Delta^+(\gg,\ga) = \{\alpha|_\ga \mid \alpha \in 
\Delta^+(\gg_\C,\gh_\C) 
\text{ and } \alpha|_\ga \ne 0\}$ and $\Delta_0^+(\gg,\ga) =
\Delta_0(\gg,\ga) \cap \Delta^+(\gg,\ga)$.
\end{itemize}
\noindent We can (and do) choose $\Delta^+(\gg,\gh)$ so that 
\begin{itemize}
\item$\gn$ is the sum of the positive restricted root spaces and
\item if $\alpha \in \Delta(\gg_\C,\gh_\C)$ and $\alpha|_\ga \in
\Delta^+(\gg,\ga)$ then $\alpha \in \Delta^+(\gg_\C,\gh_\C)$.
\end{itemize}
\medskip

Recall that two roots are {\em strongly orthogonal} if their sum and their
difference are not roots.  Then they are orthogonal.  We define
\begin{equation}\label{cascade}
\begin{aligned}
&\beta'_1 \in \Delta^+(\gg,\ga) \text{ is a maximal positive restricted root
and }\\
& \beta'_{r+1} \in \Delta^+(\gg,\ga) \text{ is a maximum among the roots of }
\Delta^+(\gg,\ga) \text{ orthogonal to all } \beta'_i \text{ with } i \leqq r
\end{aligned}
\end{equation}
Then the $\beta'_r$ are mutually strongly orthogonal.  
Note that each $\beta'_r \in \Delta_0^+(\gg,\ga)$.
This is the Kostant cascade coming down from the maximal root.  Denote
\begin{equation}\label{wrong-numbering}
\{\beta'_1, \dots , \beta'_m\}: \text{ the set of strongly orthogonal roots
constructed in (\ref{cascade}).}
\end{equation}
The enumeration (\ref{wrong-numbering}) is not appropriate for the direct
limit process, but we need it for some of the lemmas below.  For direct 
limit considerations we will use the reversed ordering
\begin{equation}\label{right-numbering}
\beta_r = \beta'_{m-r+1} \text{, so the ordered sets }
\{\beta_1, \dots , \beta_m\} = \{\beta'_m, \dots , \beta'_1\}.
\end{equation}

For $1\leqq r \leqq m$ define 
\begin{equation}\label{layers}
\begin{aligned}
&\Delta^+_m = \{\alpha \in \Delta^+(\gg,\ga) \mid \beta_m - \alpha \in \Delta^+(\gg,\ga)\} 
\text{ and }\\
&\Delta^+_{m-r-1} = \{\alpha \in \Delta^+(\gg,\ga) \setminus 
	(\Delta^+_m \cup \dots \cup \Delta^+_{m-r})
	\mid \beta_{m-r-1} - \alpha \in \Delta^+(\gg,\ga)\}.
\end{aligned}
\end{equation} 

\begin{lemma} \label{fill-out} {\rm \cite[Lemma 6.3]{W2013}}
If $\alpha \in \Delta^+(\gg,\ga)$ then either 
$\alpha \in \{\beta_1, \dots , \beta_m\}$
or $\alpha$ belongs to exactly one of the sets $\Delta^+_r$\,.
In particular the Lie algebra $\gn$ of $N$ is the
vector space direct sum of its subspaces
\begin{equation}\label{def-m}
\gl_r = \gg_{\beta_r} + {\sum}_{\Delta^+_r}\, \gg_\alpha 
\text{ for } 1\leqq r\leqq m
\end{equation}
\end{lemma}

\begin{lemma}\label{layers2}{\rm \cite[Lemma 6.4]{W2013}}
The set $\Delta^+_r\cup \{\beta_r\}  
= \{\alpha \in \Delta^+ \mid \alpha \perp \beta_i \text{ for } i > r
\text{ and } \langle \alpha, \beta_r\rangle > 0\}.$
In particular, $[\gl_r,\gl_s] \subset \gl_t$ where $t = \max\{r,s\}$.
Thus $\gn$ has an increasing foliation based on the ideals
\begin{equation}\label{def-filtration}
\gl_{r,m} = \gl_{r+1} + \dots + \gl_m \text{ for } 0 \leqq r < m
\end{equation}
with a corresponding group level decomposition by normal subgroups 
$L_{r,m}$ where
\begin{equation}\label{def-filtration-group}
N = L_{0,m} = L_1L_2\dots L_m \text{ and } L_{r,m} = L_{r+1} \ltimes N_{r+1,m}
	\text{ for } 0 \leqq r < m.
\end{equation}
\end{lemma}

The structure of $\Delta^+_r$, and later of $\gl_r$, is exhibited by a 
particular Weyl group element of $\Delta(\gg,\ga)$ and the negative of
that Weyl group element.  Denote
\begin{equation}\label{beta-reflect}
s_{\beta_r} \text{ is the Weyl group reflection in } \beta_r
\text{ and } \sigma_r: \Delta(\gg,\ga) \to \Delta(\gg,\ga) \text{ by }
\sigma_r(\alpha) = -s_{\beta_r}(\alpha).
\end{equation}
Here $\sigma_r(\beta_s) = -\beta_s$ for $s \ne r$, $+\beta_s$ if $s = r$.
If $\alpha \in \Delta^+_r$ we still have $\sigma_r(\alpha) \perp \beta_i$
for $i > r$ and $\langle \sigma_r(\alpha), \beta_r\rangle > 0$.  If
$\sigma_r(\alpha)$ is negative then $\beta_r - \sigma_r(\alpha) > \beta_r$
contradicting the maximality property of $\beta_{m-r+1}$.  Thus, using 
Lemma \ref{layers2}, $\sigma_r(\Delta^+_r) = \Delta^+_r$.
This divides each $\Delta^+_r$ into pairs:

\begin{lemma} \label{layers-nilpotent}{\rm \cite[Lemma 6.8]{W2013}}
If $\alpha \in \Delta^+_r$ then $\alpha + \sigma_r(\alpha) = \beta_r$.
{\rm (}Of course it is possible that 
$\alpha = \sigma_r(\alpha) = \tfrac{1}{2}\beta_r$ when 
$\tfrac{1}{2}\beta_r$ is a root.{\rm ).}
If $\alpha, \alpha' \in \Delta^+_r$ and $\alpha + \alpha' \in \Delta(\gg,\ga)$
then $\alpha + \alpha' = \beta_r$\,.
\end{lemma}

It comes out of Lemmas \ref{fill-out} and \ref{layers2} that the 
decompositions of (\ref{layers}), (\ref{def-m}) and
(\ref{def-filtration}) satisfy (\ref{newsetup}), so
Theorem \ref{plancherel-general}
applies to nilradicals of minimal parabolic subgroups.  In other words,
as in Theorem \ref{plancherel-general},
\begin{theorem}\label{iwasawa-layers}{\rm \cite[Theorem 6.16]{W2013}}
Let $G$ be a real reductive Lie group, $G = KAN$ an Iwasawa
decomposition, $\gl_r$ and $\gn_r$ the subalgebras of $\gn$ defined in 
{\rm (\ref{def-m})} and {\rm (\ref{def-filtration})},
and $L_r$ and $N_r$ the corresponding analytic subgroups of $N$.  
Then the $L_r$ and $N_r$ satisfy {\rm (\ref{newsetup})}.  In particular,
Plancherel measure for $N$ is
concentrated on $\{\pi_\lambda \mid \lambda \in \gt^*\}$.
If $\lambda \in \gt^*$, and if $u$ and $v$ belong to the
representation space $\cH_{\pi_\lambda}$ of $\pi_\lambda$,  then
the coefficient $f_{u,v}(x) = \langle u, \pi_\lambda(x)v\rangle$
satisfies
\begin{equation}
||f_{u,v}||^2_{L^2(N / S)} = \frac{||u||^2||v||^2}{|\Pf(\lambda)|}\,.
\end{equation}
The distribution character $\Theta_{\pi_\lambda}$ of $\pi_{\lambda}$ satisfies
\begin{equation}
\Theta_{\pi_\lambda}(f) = c^{-1}|\Pf(\lambda)|^{-1}\int_{\cO(\lambda)}
        \widehat{f_1}(\xi)d\nu_\lambda(\xi) \text{ for } f \in \cC(N)
\end{equation}
where $\cC(N)$ is the Schwartz space, $f_1$ is the lift
$f_1(\xi) = f(\exp(\xi))$, $\widehat{f_1}$ is its classical Fourier transform,
$\cO(\lambda)$ is the coadjoint orbit $\Ad^*(N)\lambda = \gv^* + \lambda$,
$c = 2^{d_1 + \dots + d_m} d_1! d_2! \dots d_m!$ as in {\rm (\ref{c-d}a)},
and $d\nu_\lambda$ is the translate of normalized Lebesgue measure from
$\gv^*$ to $\Ad^*(N)\lambda$.  The Fourier inversion formula on $N$ is
\begin{equation}
f(x) = c\int_{\gt^*} \Theta_{\pi_\lambda}(r_xf) |\Pf(\lambda)|d\lambda
        \text{ for } f \in \cC(N).
\end{equation}
\end{theorem}

\section{Nilradicals of Parabolics in Infinite Dimensional Groups}
\label{sec8}
\setcounter{equation}{0}
We now look at the classical real forms of the three classical simple locally 
finite countable--dimensional Lie
algebras $\gg_\C = \varinjlim \gg_{n,\C}$, and their real forms
$\gg_\R$.  The Lie algebras $\gg_\C$ are the classical direct limits,
$\gsl(\infty,\C) = \varinjlim \gsl(n;\C)$,
$\gso(\infty,\C) = \varinjlim \gso(2n;\C) = \varinjlim \gso(2n+1;\C)$, and
$\gsp(\infty,\C) = \varinjlim \gsp(n;\C)$,
where the direct systems are
given by the inclusions of the form
$A \mapsto (\begin{smallmatrix} A & 0 \\ 0 & 0 \end{smallmatrix} )$
or $A \mapsto \left (\begin{smallmatrix} 0 & 0 & 0 \\ 0 & A & 0 \\ 0 & 0 & 0 
\end{smallmatrix} \right )$.
We often consider the locally reductive algebra
$\ggl(\infty;\C) = \varinjlim \ggl(n;\C)$ along with $\gsl(\infty;\C)$.
\medskip

Let $G_n$ be a real (this includes complex) simple Lie group of classical 
type and real rank $n$.  We have just described it as sitting in a direct 
system $\{G_n\}$ of Lie algebras in the same series.  
Set $G = \varinjlim G_n$ as above.  Then we have coherent Iwasawa
decompositions $G_n = K_nA_nN_n$ with $K_n \subset K_\ell$, 
$A_n \subset A_\ell$ and $N_n \subset N_\ell$ for $\ell \geqq n$.  We need
to do this so that the direct limit respects the restricted root structures,
in particular the strongly orthogonal root structures, 
of the $N_n$\,.   To do that we enumerate the set 
$\Psi_n = \Psi(\gg_n, \gh_n)$ of nonmultipliable simple restricted
roots so that, in the Dynkin diagram, for type $A$ we spread from the 
center of the diagram.  For types $B$, $C$ and $D$ 
$\psi_1$ is the \textit{right} endpoint,
In other words for $\ell \geqq n$ $\Psi_\ell$
is constructed from $\Psi_n$ adding simple roots to the \textit{left} end
of their Dynkin diagrams.  Thus
\begin{equation}\label{rootorderA}
\begin{aligned} 
&\begin{tabular}{|c|l|c|}\hline
$\Psi_\ell \text{ type } A_{2\ell+1}$ &
\setlength{\unitlength}{.4 mm}
\begin{picture}(180,18)
\put(10,2){\circle{2}}
\put(5,5){$\psi_{-\ell}$}
\put(11,2){\line(1,0){13}}
\put(27,2){\circle*{1}}
\put(30,2){\circle*{1}}
\put(33,2){\circle*{1}}
\put(36,2){\line(1,0){13}}
\put(50,2){\circle{2}}
\put(45,5){$\psi_{-n}$}
\put(51,2){\line(1,0){13}}
\put(67,2){\circle*{1}}
\put(70,2){\circle*{1}}
\put(73,2){\circle*{1}}
\put(76,2){\line(1,0){13}}
\put(90,2){\circle{2}}
\put(87,5){$\psi_0$}
\put(91,2){\line(1,0){13}}
\put(107,2){\circle*{1}}
\put(110,2){\circle*{1}}
\put(113,2){\circle*{1}}
\put(116,2){\line(1,0){13}}
\put(130,2){\circle{2}}
\put(128,5){$\psi_n$}
\put(131,2){\line(1,0){13}}
\put(147,2){\circle*{1}}
\put(150,2){\circle*{1}}
\put(153,2){\circle*{1}}
\put(156,2){\line(1,0){13}}
\put(170,2){\circle{2}}
\put(167,5){$\psi_\ell$}
\end{picture}
&$\ell \geqq n \geqq 0$
\\
\hline
\end{tabular}\\
&\begin{tabular}{|c|l|c|}\hline
\setlength{\unitlength}{.4 mm}
$\Psi_\ell \text{ type } A_{2\ell}\phantom{i.}$ &
\setlength{\unitlength}{.4 mm}
\begin{picture}(180,18)
\put(1,2){\circle{2}}
\put(-4,5){$\psi_{-\ell}$}
\put(2,2){\line(1,0){13}}
\put(18,2){\circle*{1}}
\put(21,2){\circle*{1}}
\put(24,2){\circle*{1}}
\put(27,2){\line(1,0){13}}
\put(41,2){\circle{2}}
\put(36,5){$\psi_{-n}$}
\put(42,2){\line(1,0){13}}
\put(58,2){\circle*{1}}
\put(61,2){\circle*{1}}
\put(64,2){\circle*{1}}
\put(67,2){\line(1,0){13}}
\put(81,2){\circle{2}}
\put(78,5){$\psi_{-1}$}
\put(82,2){\line(1,0){13}}
\put(96,2){\circle{2}}
\put(93,5){$\psi_1$}
\put(97,2){\line(1,0){13}}
\put(113,2){\circle*{1}}
\put(116,2){\circle*{1}}
\put(119,2){\circle*{1}}
\put(122,2){\line(1,0){13}}
\put(136,2){\circle{2}}
\put(134,5){$\psi_n$}
\put(137,2){\line(1,0){13}}
\put(153,2){\circle*{1}}
\put(156,2){\circle*{1}}
\put(159,2){\circle*{1}}
\put(162,2){\line(1,0){13}}
\put(176,2){\circle{2}}
\put(173,5){$\psi_\ell$}
\end{picture}
&$\ell \geqq n \geqq 1$
\\
\hline
\end{tabular}
\end{aligned}
\end{equation}

\begin{equation}\label{rootorderBCD}
\begin{aligned}
&\begin{tabular}{|c|l|c|}\hline
$\Psi_\ell \text{ type } B_\ell$&
\setlength{\unitlength}{.5 mm}
\begin{picture}(155,13)
\put(5,2){\circle{2}}
\put(2,5){$\psi_{\ell}$}
\put(6,2){\line(1,0){13}}
\put(24,2){\circle*{1}}
\put(27,2){\circle*{1}}
\put(30,2){\circle*{1}}
\put(34,2){\line(1,0){13}}
\put(48,2){\circle{2}}
\put(45,5){$\psi_n$}
\put(49,2){\line(1,0){23}}
\put(73,2){\circle{2}}
\put(70,5){$\psi_{n-1}$}
\put(74,2){\line(1,0){13}}
\put(93,2){\circle*{1}}
\put(96,2){\circle*{1}}
\put(99,2){\circle*{1}}
\put(104,2){\line(1,0){13}}
\put(118,2){\circle{2}}
\put(115,5){$\psi_2$}
\put(119,2.5){\line(1,0){23}}
\put(119,1.5){\line(1,0){23}}
\put(143,2){\circle*{2}}
\put(140,5){$\psi_1$}
\end{picture}
&$\ell\geqq n \geqq 2$\\
\hline
\end{tabular} \\
&\begin{tabular}{|c|l|c|}\hline
$\Psi_\ell \text{ type } C_\ell$ &
\setlength{\unitlength}{.5 mm}
\begin{picture}(155,13)
\put(5,2){\circle*{2}}
\put(2,5){$\psi_{\ell}$}
\put(6,2){\line(1,0){13}}
\put(24,2){\circle*{1}}
\put(27,2){\circle*{1}}
\put(30,2){\circle*{1}}
\put(34,2){\line(1,0){13}}
\put(48,2){\circle*{2}}
\put(45,5){$\psi_n$}
\put(49,2){\line(1,0){23}}
\put(73,2){\circle*{2}}
\put(70,5){$\psi_{n-1}$}
\put(74,2){\line(1,0){13}}
\put(93,2){\circle*{1}}
\put(96,2){\circle*{1}}
\put(99,2){\circle*{1}}
\put(104,2){\line(1,0){13}}
\put(118,2){\circle*{2}}
\put(115,5){$\psi_2$}
\put(119,2.5){\line(1,0){23}}
\put(119,1.5){\line(1,0){23}}
\put(143,2){\circle{2}}
\put(140,5){$\psi_1$}
\end{picture}
& $\ell\geqq n \geqq 3$
\\
\hline
\end{tabular}\\
&\begin{tabular}{|c|l|c|}\hline
$\Psi_\ell \text{ type } D_\ell$ &
\setlength{\unitlength}{.5 mm}
\begin{picture}(155,20)
\put(5,9){\circle{2}}
\put(2,12){$\psi_{\ell}$}
\put(6,9){\line(1,0){13}}
\put(24,9){\circle*{1}}
\put(27,9){\circle*{1}}
\put(30,9){\circle*{1}}
\put(34,9){\line(1,0){13}}
\put(48,9){\circle{2}}
\put(45,12){$\psi_n$}
\put(49,9){\line(1,0){23}}
\put(73,9){\circle{2}}
\put(70,12){$\psi_{n-1}$}
\put(74,9){\line(1,0){13}}
\put(93,9){\circle*{1}}
\put(96,9){\circle*{1}}
\put(99,9){\circle*{1}}
\put(104,9){\line(1,0){13}}
\put(118,9){\circle{2}}
\put(113,12){$\psi_3$}
\put(119,8.5){\line(2,-1){13}}
\put(133,2){\circle{2}}
\put(136,0){$\psi_1$}
\put(119,9.5){\line(2,1){13}}
\put(133,16){\circle{2}}
\put(136,14){$\psi_2$}
\end{picture}
& $\ell\geqq n \geqq 4$
\\
\hline
\end{tabular}
\end{aligned}
\end{equation}
We describe this by saying that $G_\ell$ {\em propagates} $G_n$\,.
For types $B$, $C$ and $D$ this is the same as the notion of propagation in
\cite{OW2011} and \cite{OW2014}, but for type $A$ is it s bit different.
With the simple root enumeration of (\ref{rootorderA}) and (\ref{rootorderBCD})
the set $\{\beta_1, \dots , \beta_m\}$ of strongly orthogonal positive 
restricted roots of (\ref{right-numbering}) is
\medskip

type $A_{2n+1}$: $m = n+1$; $\beta_1 = \psi_0$\,; 
	$\beta_2 = \psi_{-1} + \psi_0 + \psi_1$\,; $\cdots$ ;\,
	$\beta_r = \psi_{-r+1} + \beta_{r-1} + \psi_{r-1}$\,; $\cdots$
\medskip

type $A_{2n}$: $m = n$; $\beta_1 = \psi_{-1} + \psi_1$\,;
	$\beta_2 = \psi_{-2} + \psi_{-1} + \psi_1 + \psi_2$\,; $\cdots$ ;\, 
	$\beta_r = \psi_{-r} + \beta_{r-1} + \psi_r$\,; $\cdots$
\medskip

type $B_{2n+1}$: $m = 2n+1$; $\beta_1 = \psi_1\,;
	\beta_2 = \psi_3$ and $\beta_3 = 2(\psi_1 + \psi_2)+\psi_3$\,;
	$\cdots$ ; \hfill\newline
	\phantom{XXXXXXXXXXXXXXXXXXXXXXx}$\beta_{2r} = \psi_{2r+1}$ and
	$\beta_{2r+1} = 2(\psi_1 + \dots \psi_{2r})+\psi_{2r+1}$\,;\,\,$\cdots$
\medskip

type $B_{2n}$: $m = 2n$; $\beta_1 = \psi_2$ and $\beta_2 = 2\psi_1 + \psi_2$\,; 
   $\beta_3 = \psi_4$ and $\beta_4 = 2(\psi_1 + \psi_2 + \psi_3) + \psi_4$\,;
   $\cdots$ ; \hfill\newline
   \phantom{XXXXXXXXXXXXXi}$\beta_{2r+1} = \psi_{2r-1}$ and
	$\beta_{2r} = 2(\psi_1 + \dots \psi_{2r-1}) + \psi_{2r}$\,;\,\,$\cdots$
\medskip

type $C_n$: $m = n$; $\beta_1 = \psi_1$\,; $\beta_2 = \psi_1 + 2\psi_2$\,;
	\,\,$\cdots$\,; $\beta_r = \psi_1 + 2(\psi_2 + \dots + \psi_r)$\,;
	\,\,$\cdots$
\medskip

type $D_{2n+1}$: $m = 2n$; $\beta_1 = \psi_3$\,; 
	$\beta_2 = \psi_1 + \psi_2 + \psi_3$\,; \hfill\newline
	\phantom{XXXXXXXXXXXXXXi} $\beta_3 = \psi_5$ and 
		$\beta_4 = \psi_1 + \psi_2 + 2(\psi_3 +
		\psi_4) + \psi_5$\,; $\cdots$ ; \hfill\newline
        \phantom{XXXXXXXXXXXXXXi} $\beta_{2r-1} = \psi_{2r+1}$ and 
	$\beta_{2r} = \psi_1 + \psi_2 
	  + 2(\psi_3 + \dots + \psi_{2r}) + \psi_{2r+1}$\,;\,\,$\cdots$
\medskip

type $D_{2n}$: $m = 2n$; $\beta_1 = \psi_1$\,; $\beta_2 = \psi_2$\,;
	$\beta_3 = \psi_4$ and $\beta_4 = \psi_1 + \psi_2 + 2\psi_3 + \psi_4$\;
	\hfill\newline \phantom{XXXXXXXXXXXXXXXXXXXXXXXXi}
	$\beta_5 = \psi_6$ and $\beta_6 = \psi_1 + \psi_2 + 
		2(\psi_3 + \psi_4 + \psi_5) + \psi_6$\,; $\cdots$ ; 
	\hfill\newline \phantom{XXXXXXXXXXXXXXXXXXXXXXXXi}
	$\beta_{2r-1} = \psi_{2r}$ and $\beta_{2r} = \psi_1 + \psi_2 + 
		2(\psi_3+ \dots + \psi_{2r-1})+\psi_{2r}$\,;\,\,$\cdots$
\smallskip

In order to simplify use of these constructions we denote
\begin{definition}\label{well-aligned}{\rm
Let $G = \varinjlim G_n$ be a classical simple locally finite 
countable dimensional Lie group.  Possibly passing to a cofinal subsequence
suppose that we have coherent Iwasawa decompositions $G_n = K_n A_n N_n$ 
such that $G_\ell$ propagates $G_n$ for
$\ell \geqq n$.  Then, again possibly passing to a cofinal subsequence,
we can assume that all of the nonmultipliable restricted root 
systems $\Delta_0(\gg_n,\ga_n)$ are of the same type $A_{2n+1}$,
$A_{2n}$, $B_{2n+1}$, $B_{2n}$, $C_n$, $D_{2n+1}$ or $D_{2n}$.  Then 
we will say that the direct system $\{G_n\}$ is} well--aligned.
\hfill $\diamondsuit$
\end{definition}

The condition that $\{G_n\}$ be well--aligned is exactly what we need for
$\{N_n\}$ to satisfy (\ref{newsetup}), and given $G$ we have a realization
$G = \varinjlim G_n$ for which $\{G_n\}$ is well--aligned.  In summary,
\begin{theorem}\label{inversion-for-ss-limits}
Let $G$ be a classical connected countable dimensional real reductive Lie
group.  Express $G = \varinjlim G_n$ with $\{G_n\}$ well--aligned.
Then $\{N_n\}$ satisfies {\rm (\ref{newsetup})}.  In particular
{\rm Theorem \ref{iwasawa-layers}} holds for the 
maximal locally unipotent subgroup $N = \varinjlim N_n$ of $G$.
\end{theorem}

\begin{remark}{\rm
In Theorem \ref{inversion-for-ss-limits} the possibilities for $G$ are
the finite dimensional simple Lie groups and the infinite dimensional
$SL(\infty;\C)$, $SO(\infty;\C)$, $Sp(\infty;\C)$, $SL(\infty;\R)$,
$SL(\infty;\H)$, $SU(\infty,q)$ with $q \leqq \infty$, $SO(\infty,q)$
with $q \leqq \infty$, $Sp(\infty,q)$ with $q \leqq \infty$, 
$Sp(\infty;\R)$ and $SO^*(2\infty)$.  Further, the normalizer
$P = MAN$ of $N$ in $G$ is a classical minimal parabolic subgroup
$\varinjlim (P_n = M_nA_nN_n)$ where $P_n$ is the minimal
parabolic in $G_n$ that is the normalizer of $N_n$\,.}
\hfill $\diamondsuit$
\end{remark}

\medskip
\noindent Department of Mathematics, University of California,\hfill\newline
\noindent Berkeley, California 94720--3840, USA\hfill\newline
\smallskip
\noindent {\tt jawolf@math.berkeley.edu}

\enddocument
\end